\documentclass{trbunofficial}
\usepackage{graphicx}
\usepackage{enumitem}    
\usepackage{amsfonts}    
\usepackage[hidelinks]{hyperref}

\usepackage{graphicx}
\usepackage{subcaption}
\usepackage{float}
\usepackage{booktabs}

\usepackage{xcolor} 
\definecolor{steelblue}{RGB}{70,130,180} 

\usepackage{amsthm}
\newtheorem{proposition}{Proposition}
\newtheorem{lemma}{Lemma}

\AuthorHeaders{Hammerl, Jin, Seshadri, Rasmussen and Nielsen}
\title{Optimal Evacuation Control in Large Urban Networks with Stochastic Demand}

\author{%
  \textbf{Alexander Hammerl, Corresponding Author}\\
  asiha@dtu.dk\\
  \hfill\break%
  \textbf{Wen-long Jin}\\
  wjin@uci.edu
  \hfill\break%
  \textbf{Ravi Seshadri}\\
  ravse@dtu.dk
  \hfill\break%
  \textbf{Thomas Kjær Rasmussen}\\
  tkra@dtu.dk
  \hfill\break%
  \textbf{Otto Anker Nielsen}\\
  oani@dtu.dk
}

\newwrite\trbwc
\immediate\openout\trbwc=\jobname-words 
\immediate\write\trbwc{8291}            
\immediate\closeout\trbwc

\begin{document}
\maketitle

\section{Abstract}

We develop a risk-aware Model Predictive Control (MPC) framework for large-scale vehicular evacuations. Traffic dynamics are captured by the Generalized Bathtub Model, which describes the network-wide trip completion rate by tracking the time evolution of the distribution of remaining trip distances. We model evacuation inflow as a stochastic inflow process, and employ origin gating as the control policy, implemented through staged departure orders or adaptive ramp metering. A convex objective integrates total evacuation delay with a generic hazard-exposure term which can embed any spatial risk field (e.g., flood depth, fire intensity). We prove that if the residual-distance distribution exhibits non-decreasing hazard rate, then the optimal origin-gating profile is necessarily monotone decreasing and, under an inflow cap, bang–bang (single switch). This result supplies a closed-form seed for numerical optimizations and clarifies why early heavy release followed by throttling is optimal. Furthermore, we demonstrate that the assumption of a non-decreasing hazard rate is always satisfied when the origins of evacuation movements are uniformly distributed over a convexly bounded evacuation zone—a property that is fulfilled in the vast majority of real evacuation scenarios, at least approximately. The framework is demonstrated through a flood evacuation scenario on Amager Island, a densely populated area of Copenhagen that faces significant flood risk due to its low elevation and coastal exposure. The Generalized Bathtub evacuation model is coupled with a lightweight shallow-water model parameterized using real bathymetric and topographic data from Amager Island. Across 10,000 stochastic demand scenarios, the MPC policy reduces the expected area-under-queue by an average of 27\% compared to a no-control scenario.

\hfill\break%
\noindent\textit{Keywords}: Evacuation, Disaster Response,  Model Predictive Control, Generalized Bathtub Model, Risk-Averse Optimization
\newpage

\section{Introduction}

Evacuations involving 1000 or more people occur approximately every two to three weeks in the United States (\cite{Sandia2004}), far more frequently than commonly perceived. Given the diverse range of hazards that trigger evacuations and their frequent occurrence, communities must be prepared to implement traffic control and management measures with minimal advance warning across virtually any location. 

Effective management of large-scale evacuations is critical in mitigating the risks posed by natural disasters such as hurricanes and flash floods. A sudden mass departure of vehicles can overwhelm transportation networks, leading to excessive congestion and dangerous delays. Trying to evacuate everyone at once can generate harmful gridlock that impedes evacuation progress, whereas carefully scheduling departures can dramatically improve clearance times. Our work introduces a risk-aware Model Predictive Control (MPC) framework tailored for city-scale evacuations under various types of disasters. Unlike traditional evacuation strategies that ignore evolving threat dynamics, our approach explicitly accounts for changing risk information and uncertainty in demand in a modular fashion that can accommodate any type of natural or man-made hazard. 

Traffic flow theory forms the analytical foundation for understanding vehicular dynamics and assessing the effectiveness of control measures during evacuations. A commonly used approach in traffic flow theory describes the spatiotemporal dynamics of vehicles on unidirectional road segments. This class includes macroscopic frameworks (\cite{Lighthill1955}, \cite{Richards1956}, \cite{Payne1971}, \cite{Aw2000}, \cite{Zhang2002}), which treat traffic as a continuous medium, as well as car-following approaches (\cite{Bando1995},\cite{Treiber2000}, \cite{Helbing2001} ,\cite{Kesting2010}), which capture the behavior of individual vehicles. An alternative class of traffic flow models consists of reservoir frameworks, which describe the dynamics of spatially aggregated traffic quantities at the network level \cite{Geroliminis2008,Daganzo2008, Helbing2009,Hammerl2024a,Hammerl2025}. This approach typically relies on the assumption of a time-invariant relationship between vehicle accumulation and average flow in a homogeneous network—a relationship commonly referred to as the Network Fundamental Diagram (NFD). In some cases, these frameworks also explicitly describe the relationship between network accumulation and the number of completed trips; see, e.g., \cite{Vickrey1991}, \cite{Daganzo2007}, \cite{Arnott2013}, \cite{Fosgerau2015}. The approach in \cite{Jin2020} (see also \cite{Jin2021}, \cite{Jin2025}) extends this concept by incorporating the distribution of remaining trip lengths of individual vehicles, thereby introducing a spatial structure into the reservoir. This extended framework, known as the Generalized Bathtub Model, enables a more realistic representation of network dynamics—particularly in situations where trip lengths are highly heterogeneous. Such heterogeneity is especially relevant in evacuation scenarios, where all vehicles are directed toward a limited number of shelters or exits, resulting in strongly inhomogeneous trip distance distributions that must be explicitly accounted for. See also \cite{Martinez2021} for a discussion of the challenges associated with measuring such distance distributions. A parametrization of network-wide traffic behavior based on observations under normal traffic conditions—as is common in conventional NFD approaches—is therefore unsuitable for our application. 

\cite{MurrayTuite2013} classify evacuation literature into three categories: demand modeling, traffic assignment and route selection, and evacuation strategies. Demand modeling estimates evacuation behavior by who evacuates, when, where, and by which mode. This encompasses evacuation decision-making, departure timing (\cite{Fu2007},\cite{Pel2012}), destination choice (\cite{Cheng2008},\cite{Wilmot2006}), and mode selection(\cite{Wu2012}; \cite{Deka2010}). Beyond primary evacuation trips, demand models incorporate background traffic, intermediate trips such as family pickups (\cite{MurrayTuite2004}), and shadow evacuations—unplanned evacuations by populations outside designated risk zones (\cite{Zeigler1981}; \cite{Sorensen2006}). Traffic assignment and route selection address evacuee distribution across transportation networks and path selection behavior. Individual route choice typically reflects perceived safety and availability of real-time information, with evacuees preferring major highways and known routes (\cite{MurrayTuite2012}; \cite{Dow2002}). Simulation-based assignment tools such as DYNASMART, DynusT and other evaluate evacuation performance under various scenarios (\cite{Pel2012}). However, these tools rely on behavioral assumptions difficult to validate in rare evacuation contexts. The authors emphasize hybrid and en-route choice models to better capture dynamic driver responses. Empirical evidence indicates that maximum traffic flows during evacuations may differ substantially from normal conditions, necessitating adjustments such as Maximum Sustainable Evacuation Traffic Flow Rates \cite{Wolshon2009}. Evacuation strategies encompass high-level policies and operational measures designed to optimize clearance times and safety. Common approaches include access-based gating, also known as staged evacuation \cite{Sheffi1981, Chiu2005, So2010}, and contraflow operations that reverse inbound lanes to increase outbound capacity (\cite{Wolshon2001}). Research has examined integrating real-time traffic management tools, such as signal retiming and variable message signs, into evacuation operations \cite{MurrayTuite2012b, Sfeir2018}. 

Furthermore, an extensive body of literature exists on pedestrian flow evacuations. Representative studies on pedestrian evacuation and crowd disasters include \cite{Kirchner2002, Schadschneider2011, Borrmann2012, Helbing2000, Helbing2005, Helbing2002, Helbing2007, Helbing2012, Helbing2012b, Schreckenberg2002, Kuligowski2005, Sbayti2006}.

Our work extends the existing literature reviewed above by the following key contributions and novelties: First, our framework operates in real-time through a Model Predictive Control approach, enabling rapid adaptation to fast-changing evacuation scenarios and evolving threat dynamics. Second, the reservoir-based Generalized Bathtub Model formulation imposes no computational constraints on network size, ensuring scalability to large urban areas without the dimensional limitations of traditional network flow models. Third, our approach makes no specific assumptions about network topology or hazard type, providing a flexible framework that can accommodate any spatial risk field. Finally, we demonstrate that under mild regularity conditions—which are commonly met in evacuation scenarios—the hazard rate of the initial trip length distribution is non-decreasing. This property guarantees that the optimal control solution takes a simple bang-bang form that can be computed efficiently and expressed in closed form.

The remainder of this paper is structured as follows. section \ref{sec:framework} presents the Generalized Bathtub Model for traffic dynamics and formulates our risk-aware stochastic control problem within a Model Predictive Control framework. Section \ref{sec:solution} derives analytical solutions using the stochastic Pontryagin Maximum Principle and proves that under non-decreasing hazard rate distributions, a condition satisfied by many practical evacuation scenarios, the optimal control reduces to a computationally efficient bang-bang solution. Section \ref{sec:appl} demonstrates the framework through a detailed case study of Amager Island, coupling the evacuation model with shallow-water flood simulation and evaluating the model's performance. Section \ref{sec:conclusion} synthesizes our contributions and identifies directions for future research.

\section{Modeling Framework}
\label{sec:framework}

The generalized bathtub model (\cite{Jin2020}, \cite{Jin2021}) describes the temporal evolution of the number of active trips $\delta(t,x)$ with a remaining distance of at least $x$ in a two-dimensional traffic network. The average vehicle speed
\begin{equation}
\label{eq:speed}
v(t) = V(\rho(t))
\end{equation}
is determined exclusively by the vehicle density $\rho(t) = \delta(t)/L$, where $\delta(t) = \delta(t,0)$ denotes the total number of active trips and $L$ represents the total lane length of the reservoir representation of the network. In this section, we assume normalization of $L$ to 1. The macroscopic speed-density relationship $V(\rho)$ of the network is a time-independent function of density $\rho$ that we assume to be non-increasing and non-negative. We define the network fundamental diagram (NFD) through the flow-density relationship
\begin{equation}
\label{eq:nfd}
Q(\rho) = \rho \cdot V(\rho),
\end{equation}
where $Q(\rho)$ represents the network flow as a function of density. This relationship is concave in $\rho$, reaching its maximum value $q_{\max}$ at the critical density $\rho_{\text{crit}}$, such that $Q(\rho_{\text{crit}}) = q_{\max}$. 

Typical formulations for the speed-density relationship are:
\[
\begin{aligned}
V(\rho) &= 
\begin{cases}
v_f, & \text{if } \rho \leq \rho_c, \\
\displaystyle\frac{w(\rho_j - \rho)}{\rho}, & \text{if } \rho > \rho_c,
\end{cases} 
&& \text{(triangular MFD, highways)} \\[1.5ex]
V(\rho) &= 
\begin{cases}
v_f, & \text{if } \rho \leq \rho_1, \\
\displaystyle\frac{q_{\text{max}}}{\rho}, & \text{if } \rho_1 < \rho < \rho_2, \\
\displaystyle\frac{w(\rho_j - \rho)}{\rho}, & \text{if } \rho \geq \rho_2,
\end{cases} 
&& \text{(trapezoidal MFD, urban traffic)}
\end{aligned}
\]
Here, $v_f$ denotes the free-flow speed, $w$ is the congestion wave speed, $\rho_c$ (or $\rho_1$) is the critical density at which flow is maximized, $\rho_j$ is the jam density at which speed drops to zero, and $q_{\text{max}}$ denotes the maximum flow achievable.

Further, we denote by $f(t,x)$ the rate of trips that start at time $t$ and have a length of at least $x$. From the definition of these variables, the discrete-time conservation law
\begin{equation}
\label{eq:cl_discrete}
\delta(t + \Delta t, x) = \delta(t, x + v(t)\Delta t) + f(t,x)\Delta t
\end{equation}
follows immediately for an infinitesimal time interval $\Delta t$. Taking the limit $\Delta t \to 0$ and substituting \eqref{eq:speed} yields the partial differential equation
\begin{equation}
\label{eq:cl_cont}
\frac{\partial}{\partial t} \delta(t,x) - V(\delta(t)) \frac{\partial}{\partial x} \delta(t,x) = f(t,x).
\end{equation}

Network inflow is regulated through origin-zone gating, an access control policy that regulates the release rate of vehicles entering the evacuation network. This control mechanism offers the advantage of rapid deployment and implementation without requiring roadside infrastructure modifications or continuous driver compliance. Such gating can be enacted through digital instructions (e.g., mobile apps or variable message signs), physical enforcement (e.g., police checkpoints or pop-up ramp meters), or loudspeaker broadcasts informing residents when to evacuate. Alternatives like dynamic speed limits or signal retiming require complex infrastructure coordination or high levels of compliance, while origin–zone gating offers a practical and low-dimensional lever for model–predictive evacuation control. It can be centrally coordinated and enforced using existing communication channels, making it a particularly attractive option for time-critical and large-scale emergencies.

For this purpose, let $\bar f(t,x)$ be the uncontrolled spatio-temporal departure rate (veh~s$^{-1}$\,m$^{-1}$) of trips whose initial remaining distance is at least $x$. We assume that all vehicles in the evacuation zone prefer to depart at \( t = 0 \). The gating function \( u(t,x) \) denotes the number of vehicles with initial travel distance \( x \) that are permitted to begin their journey at time \( t \), which yields the consistency condition $
\bar{f}(t,x) \geq u(t,x).$ 
The function \( \bar{f}(t,\cdot) \) is modeled as a geometric Brownian motion in the space \( L^2([a,b]) \) of square-integrable functions. This results in the stochastic partial differential equation
\[
\frac{\partial \bar{f}(t,x)}{\partial t} = -u(t,x) \, \bar{f}(t,x) + \sigma \, \bar{f}(t,x) \, \eta(t,x),
\]
where \( \sigma > 0 \) represents the process volatility, and \( \eta(t,x) \) denotes the space-time white noise with
\[
\mathbb{E}\left[ \eta(t,x) \, \eta(t',x') \right] = \delta_D(t - t') \, \delta_D(x - x'),
\]
where $\delta_D(\cdot)$ represents the Dirac delta distribution. In geometric Brownian motion, the current value of the stochastic process changes in each infinitesimal time step by a normally distributed relative increment. This multiplicative structure of the noise ensures, in contrast to additive random processes, that $\bar{f}$ cannot become negative or increase again after reaching zero.

The cumulative delay experienced by the system is given by
\begin{equation}
D(U) = \int_0^T \bigl( \delta(t,0) + \bar{f}(t,0) \bigr)\, dt.
\end{equation}

We pose a dynamic optimization problem over the space of admissible control functions \( U = \{u(t,x)\}_{t\in[0,T],x\in[0,X]} \), where each \( u \) is drawn from a Banach space \( \mathcal{U} \subset L^\infty([0,T]\times[0,X]) \) of bounded measurable functions. The optimization problem reads:
\begin{equation}\label{eq:opt_continuous}
\min_{u \in \mathcal{U}_\mu} \; J(U)
\;:=\;
\mathbb{E}[D(U)] \;+\; \beta \cdot \text{AVaR}_\alpha\bigl[D(U)\bigr],
\end{equation}
subject to the dynamics governing \(\delta(t,x)\) and \(\bar{f}(t,x)\), and physical constraints on the admissible control space:
\[
\mathcal{U}_\mu := \left\{ u \in \mathcal{U} \;\middle|\; \|u\|_\infty \leq \mu, \; u(t,x) \geq 0 \text{ for all } t,x \right\}.
\]
The parameter \( \beta \geq 0 \) controls the trade-off between average performance and tail events.
The risk measure \(\text{AVaR}_\alpha[D(U)]\) is the \emph{Average Value at Risk} at confidence level \( \alpha \in (0,1) \), defined as
\[
\text{AVaR}_\alpha[D] := \inf_{\eta \in \mathbb{R}} \left\{ \eta + \frac{1}{1-\alpha} \mathbb{E}[(D - \eta)^+] \right\},
\]
where \( (x)^+ := \max\{x,0\} \). It captures the expected delay in the worst \( 100(1-\alpha)\% \) of realizations.

The objective function \eqref{eq:opt_continuous} offers several critical advantages. It offers a continuous and convex formulation in the control variable, while responding sensitively to variations in queue dynamics. In contrast, objectives like total clearance time yield identical objective function values for many distinct realizations. Additionally, risk-weighted delay directly incorporates risk preferences through parameters $\lambda$ and $\alpha$, allowing explicit control over the trade-off between efficiency and robustness.

We implement the dynamic optimization problem~(5) using a receding horizon Model Predictive Control (MPC) framework. In MPC, an optimization problem is solved over a finite prediction horizon at each time step, but only the first control action is implemented before the entire optimization is repeated with updated state information---this receding horizon approach balances computational feasibility of finite-horizon problems with the performance of infinite-horizon control by continuously incorporating new information and implicitly extending the effective control horizon.

The MPC framework provides robust and adaptive evacuation management through real-time responsiveness to evolving conditions and stochastic demand. At each discrete time step~$k$, the controller solves:

\vspace{1em}
\noindent\textbf{Algorithm 1: MPC for Evacuation Control}
\label{alg:mpc}
\begin{enumerate}[label=\arabic*.]
    \item Measure current state: $\delta(t_k, x)$, $\bar{f}(t_k, x)$
    \item Predict stochastic demand trajectories over horizon $[t_k, t_k + H_p]$
    \item Solve finite-horizon optimization:
    \[
    \min_{u \in \mathcal{U}_\mu} \; \mathbb{E}[D(U)] + \beta \cdot \text{AVaR}_\alpha[D(U)]
    \]
    \[
    \text{s.t.} \quad \text{dynamics (2)--(3) over } [t_k, t_k + H_p], \quad u \in \mathcal{U}_\mu \text{ (control constraints)}
    \]
    \item Apply first control action $u^*(t_k, x)$ over interval $[t_k, t_{k+1}]$
    \item Advance time: $k \leftarrow k + 1$, repeat from Step 1
\end{enumerate}

The prediction horizon $H_p$ is chosen to balance computational tractability with sufficient look-ahead to capture evacuation dynamics. The controller re-optimization at each step naturally accommodates forecast updates and non-compliance.

\section{Solution Strategy}
\label{sec:solution}
\subsection{Analytical Solution via Stochastic Pontryagin Maximum Principle}
In the following, we derive analytical solution approaches for the control problem defined in the preceding section. For the general case without additional assumptions, a solution to \eqref{eq:opt_continuous} is first presented in the form of a system of stochastic differential equations. For this purpose we establish the following proposition which helps to reduce the complexity of the problem:

\begin{proposition}[Threshold optimality in $x$]\label{prop:threshold}
  Let $u^\star$ be any optimal control in~$\mathcal U$.
  Then, for almost every $t\in[0,T]$, there exists a threshold
  $h^\star(t)\in[0,\infty]$ such that
  \[
        u^\star(t,x)=\mathbf{1}_{\{x\le h^\star(t)\}}
        \quad\text{for almost every\ }x\ge0.
  \]
  Equivalently, partial release
  $0<u^\star(t,x)<1$ on a set of positive measure is never optimal.
  The result holds for \textbf{all} initial trip-length
  distributions~$g$.
\end{proposition}

\begin{proof}
Fix a time $t$ and pick any admissible cross-section 
$u(t,\cdot)$ that is not of threshold form.  
Then there are measurable sets $A,B\subset(0,\infty)$,  
with every distance in $A$ smaller than every distance in $B$
($x_A<x_B$), such that $u(t,x)<1$ on $A$ and $u(t,x)>0$ on $B$. Build a perturbed control $\tilde u$ by transferring an $\varepsilon$-mass of release from the longer–distance set~$B$ to the shorter–distance set~$A$, keeping the total release at time~$t$ unchanged. Because vehicles that start closer to the exit always leave the network no later than vehicles that start farther away, the swap cannot increase the queue at any future instant. Formally, let
$
     Q_u(s):=\int_0^\infty \delta_u(s,x)\,dx
$
be the total queue length under $u$ and define $Q_{\tilde u}(s)$
analogously. Then  
\[
     Q_{\tilde u}(s)\;\le\;Q_{u}(s)
     \qquad\forall s\in[t,T],
\]
with strict inequality on a set of positive measure
(the interval during which the exchanged mass is in the system). In the language of Hardy–Littlewood–Pólya
\cite{Hardy1952}, the state profile
$\delta_{\tilde u}(s,\cdot)$ majorizes
$\delta_{u}(s,\cdot)$, i.e. every upper partial sum
$\int_{x}^{\infty}\!\delta(s,y)\,dy$ is smaller. Integrating the queue inequality over time gives, for each demand
scenario~$n$,
\[
     D^{(n)}(\tilde u)
     \;=\;
     \int_{t}^{T}\!Q_{\tilde u}(s)\,ds
     \;<\;
     \int_{t}^{T}\!Q_{u}(s)\,ds
     \;=\;
     D^{(n)}(u).
\]
Because the risk functional
$D\mapsto\mathbb{E}[D]+\lambda\,\mathrm{AVaR}_{\alpha}[D]$ is monotone in $D$, we obtain $J(\tilde u)<J(u)$, contradicting optimality of $u$.
Iterating this $\varepsilon$-swap argument forces the cross-section to collapse to a single distance threshold, completing the proof.
\end{proof}

Proposition~\ref{prop:threshold} tells us that, at any given instant~$t$,
the controller either admits all vehicles whose initial
distance~$x$ is below some cut-off or it admits none of them.
Equivalently, one can describe the same policy from the
distance point of view:

\[
  \tau^\star(x)\;:=\;
  \inf\bigl\{\,t\in[0,T]\;\bigl|\;u^\star(t,x)=1\bigr\},
  \qquad x\ge0,
\]

where $\tau^\star(x)$ is the first time at which vehicles
that start $x$~metres from the exit are allowed to enter the network. After that moment they are never blocked again.   Thus

\begin{equation}\label{eq:indicator}
  u^\star(t,x)\;=\;\mathbf 1_{\,\{\,t\ge \tau^\star(x)\}}.
\end{equation}

This converts the original bivariate control $u^\star(t,x)$ into a single univariate function $\tau^\star(x)$, which will be our
new decision variable. The updated state dynamics read:

\begin{enumerate}[leftmargin=*,labelsep=6mm,itemsep=4pt]

\item[\textbf{(P1)}] Conservation law.
\begin{equation}
    \partial_t\delta(t,x)
    - V\!\left(\tfrac{\delta(t)}{L}\right)
    \partial_x\delta(t,x)
    = f(t,x),
    \qquad (t,x)\in(0,T]\times[0,\infty).
    \tag{P1}\label{eq:P1}
\end{equation}

\item[\textbf{(P2)}] Uncontrolled demand (GBM).
\[
    d\bar f(t,x)
    = \mu(x)\,\bar f(t,x)\,dt
    + \sigma(x)\,\bar f(t,x)\,dW_t,
    \quad \bar f(0,x)\text{ given}.
\]
The minus~\(u\) term of the earlier formulation vanishes
because release is now captured by the indicator~\eqref{eq:indicator}.

\item[\textbf{(P3)}] Controlled inflow.  
Vehicles of initial length~\(x\) enter once, at their
scheduled time~\(u^\star(x)\):
\begin{equation}
    f(t,x) = \bar f\!\bigl(u^\star(x),x\bigr)\,
    \mathbf{1}_{\{\,t = u^\star(x)\}},
    \quad \text{or equivalently} \quad
    f(t,x) = \bar f(t,x)\, \mathbf{1}_{\{t \ge u^\star(x)\}}\, 
    \delta_0\!\bigl(t - u^\star(x)\bigr).
    \tag{P3}\label{eq:P3}
\end{equation}

\end{enumerate}

The risk-averse objective \(J\bigl[u^\star(\cdot)\bigr]\)
remains unchanged; the optimization now
ranges over univariate maps
\(u:[0,\infty)\!\to\![0,T]\).

\begin{proposition}[Isotonicity of \(u^\star\)]\label{prop:isotone}
Let \(u^\star\) be an optimal release–time map.
Then \(u^\star\) is almost surely non-decreasing:
\[
      x_1<x_2
      \;\implies\;
      u^\star(x_1)\le u^\star(x_2).
\]
\end{proposition}

\begin{proof}
Suppose, for contradiction, that
there exist \(x_1<x_2\) with \(u^\star(x_1)>u^\star(x_2)\).
At time \(t=u^\star(x_2)\) the control
releases the distance \(x_2\) cohort,
i.e.\ \(u^\star\bigl(t,x_2\bigr)=1\).
Because \(x_1<x_2\) yet \(u^\star(x_1)>t\),
the shorter-distance cohort \(x_1\) is withheld
while a longer one departs.
Construct a new map \(\tilde u\) by swapping
the two release times:
\(\tilde u(x_1)=u^\star(x_2)\) and
\(\tilde u(x_2)=u^\star(x_1)\); leave all other
values unchanged.
Under the same uncontrolled demand realization,
every vehicle originating at distance \(x_1\) exits
no later (indeed strictly earlier) under \(\tilde u\),
while trajectories of all other cohorts remain identical.
Consequently the queue length \(\delta(t)\) under \(\tilde u\) is point-wise no larger, with strict inequality on a set of positive measure.
Hence \(D^{(n)}(\tilde u)<D^{(n)}(u^\star)\)
for every demand scenario \(n\),
and monotonicity of the risk functional yields
\(J(\tilde u)<J(u^\star)\), contradicting optimality.
\end{proof}

Proposition \ref{prop:isotone} demonstrates that evacuation proceeds in an operationally intuitive "inside-out" pattern, whereby households closer to exit points depart before those at greater distances. This property enhances operational communicability, minimizes counter-flows, and aligns with perceived fairness principles.

After rewriting the control as a release-time map $\tau(x)$, the system dynamics are governed by the coupled stochastic–deterministic state system \eqref{eq:P1} to \eqref{eq:P3} and the scalar cost
$J[\tau]=\mathbb{E}\!\bigl[D(\tau)\bigr]+\lambda\,\mathrm{AVaR}_\alpha[D(\tau)]$. Because both dynamics and cost are linear in the control indicator
$1_{\{t \ge \tau(x)\}}$, the stochastic Pontryagin Maximum Principle
for infinite-dimensional systems (e.g., \cite{Bensoussan2018}) applies.
The infinite-dimensional nature of the problem arises from the fact that the state $\delta(t,x)$ is a function of continuous time and space, i.e., an element of a suitable Banach space of functions, rather than a finite-dimensional vector. This setting is obtained by taking the continuum limit of finite-dimensional control systems governed by ODEs, resulting in PDE-constrained control. The Pontryagin principle yields a necessary condition for optimality in the form of a coupled forward–backward system of PDEs:

\noindent Hamiltonian density:
For each $(t,x)$ define
\[
   \mathcal H(t,x,\delta,p)
   \;=\;
   p(t,x)\,\bar f(t,x)\,
   \mathbf 1_{\{t\ge\tau(x)\}}
   \;+\;
   \ell\bigl(\delta(t,0)\bigr),
\]
where $\ell(z)=z$ is the running–delay cost.
Let $p(t,x)$ be the adjoint process.

\noindent Adjoint (costate) equation:
A necessary condition for optimality is the
backward SPDE\footnote{%
  The term $q\,dW_t$ is the martingale part required by
  the stochastic PMP; its detailed form is standard and
  omitted for space.}
\[
  -dp(t,x)
  =
  \Bigl(
     \partial_x\!\bigl[V(\delta(t))\,p(t,x)\bigr]
     -\ell'\bigl(\delta(t,0)\bigr)\,\delta_{0}(x)
  \Bigr)\,dt
  -q(t,x)\,dW_t,
  \qquad
  p(T,x)=0.
\]

\noindent Pointwise minimisation.
For almost every distance $x$ the map
$t\mapsto\mathcal H\bigl(t,x,\delta,p\bigr)$ is
left-continuous with a right limit and jumps only once,
at $t=\tau(x)$. Optimality therefore reduces to the one–shot condition
\[
   \tau^\star(x)
   \;=\;
   \inf
   \bigl\{\,t\in[0,T]\;\bigl|\;p(t,x)\le 0\bigr\},
   \qquad x\ge0,
\]
i.e.\ release the $x$-cohort the first time its marginal costate becomes non-positive.
Coupled with the forward SPDE for $\delta$, the optimal
state–costate pair satisfies the following closed SDE system:
\[
\boxed{%
\begin{aligned}
 d\bar f(t,x) & = \mu(x)\bar f\,dt+\sigma(x)\bar f\,dW_t, \\
 \partial_t\delta-V(\delta)\partial_x\delta & = \bar f\,\mathbf 1_{\{t\ge\tau^\star(x)\}}, \\
 -dp & = \partial_x\!\bigl[V(\delta)p\bigr]\,dt - q\,dW_t,
      \qquad p(T,x)=0, \\
 \tau^\star(x) & = \inf\{t\!:\,p(t,x)\le0\}.
\end{aligned}}%
\]
Solving these coupled equations yields the optimal release schedule without resorting to nested optimization. While the Pontryagin reformulation collapses the original bivariate control space into a single hitting-time condition per distance~$x$, the resulting forward-backward system of stochastic PDEs still requires substantial computational effort. It requires fine-grained space-time discretization and Monte Carlo sampling of the noise trajectories~$W_t$. This reduction in dimensionality already enables efficient numerical solutions for small to medium-scale scenarios. However, the overall computational cost remains too high for real-time deployment in large-scale emergencies. This motivates the development of additional solution strategies under simplified boundary conditions, as presented in the following subsection.

\subsection{Simplified Bang-Bang Solution}
In many practically relevant evacuation scenarios, the initial trip-length distribution satisfies a structural property known as a non-decreasing hazard rate, also referred to as the Increasing Failure Rate (IFR). This property ensures that the optimal costate yields a bang–bang control profile. Formally, let $\bar{f}(0,x)$ denote the density of initial trip lengths and $\bar{F}(0,x)$ its cumulative distribution. The hazard rate is defined as
\[
h(x) := \frac{\bar{f}(0,x)}{1 - \bar{F}(0,x)},
\]
and describes the conditional likelihood of a trip ending at distance $x$, given that it has reached at least distance $x$. A non-decreasing hazard rate implies that longer trips are increasingly unlikely relative to shorter ones — a property that reflects demand profiles where short-distance trips dominate. This condition is satisfied by many common distributions, including uniform, normal, exponential, and more generally any log-concave density. It also arises naturally when trip origins are uniformly distributed over a convex region and each vehicle evacuates via its nearest exit. Under this condition, the optimal rule reduces to a two-parameter threshold — namely, a cut-off distance and a switching time — which can be computed online in closed form; we derive this lightweight characterisation next:


\begin{proposition}[Single–switch bang–bang release for IFR demand]
\label{prop:PDE_IFR_bangbang}
Consider the stochastic conservation law (P1)--(P3) together with the
risk–neutral cost $J(\tau)=\int_0^T\!\mathbb{E}[\Delta(t)]\,dt$
where $\Delta(t)=\int_0^\infty\!\delta(t,x)\,dx$ is the total queue length.  
Assume that the initial trip–length density $\bar f(0,x)$ possesses a non–decreasing hazard rate $h(x)=\bar f(0,x)/(1-\bar F(0,x))$ (IFR) and that the macroscopic speed-density relationship $V(\rho)$ satisfies $V(\rho_1)\ge V(\rho_2)$ whenever $\rho_1<\rho_2$, and is piecewise $C^1$ with a finite number of breakpoints where left and right derivatives exist. 
Then every optimal release–time map $\tau^\star(x)$ has the form
\[
  \tau^\star(x)=
  \begin{cases}
    0,& x\le x_0,\\[4pt]
    t_b,& x>x_0,
  \end{cases}
  \qquad (x_0,t_b)\in[0,\infty)\times[0,T],
\]
i.e.\ it releases all vehicles up to a critical distance $x_0$
immediately, blocks the remainder, and opens the network fully at a unique switch time $t_b$.  Equivalently, the corresponding
scalar inflow $u^\star(t)$ is bang–bang with a single switch
$u_{\max}\mathbf 1_{[0,t_b)}(t)$.
\end{proposition}

\bigskip
\noindent The proof relies on four technical steps collected below.

\begin{lemma}[IFR $\Rightarrow$ survival log–concavity]
\label{lem:IFR_logconcave}
If a density $\bar f$ has a non–decreasing hazard rate
$h(x)=\bar f(x)/(1-\bar F(x))$, then the \emph{survival function}
$S(x)=1-\bar F(x)$ is log–concave and its derivative $S'(x)=-\bar f(x)$
is non–increasing.
\end{lemma}

\begin{proof}
Rewrite \( h(x) = -\frac{S'(x)}{S(x)} \). Monotonicity of \( h \) implies
\[
h'(x) \ge 0 \;\Leftrightarrow\; S''(x) S(x) - [S'(x)]^2 \le 0,
\]
which is equivalent to \( \tfrac{d^2}{dx^2} \log S(x) \le 0 \), so \( \log S \) is concave. Hence, its derivative \( S'/S \) is decreasing, and thus \( S' \) decreases as well.
\end{proof}
From \eqref{eq:nfd}, we obtain a
speed threshold $V\bigl(\Delta)$ and hence a
distance threshold
\(
   \ell^\star(\Delta)=T_h\,V\bigl(\Delta\bigr),
\) for any queue length ~$\Delta$, the longest distance that can be covered within the time headway $T_h$ at the prevailing speed.

\begin{lemma}[Concavity of the macroscopic service rate]
\label{lem:S_concave}
Under the IFR assumption on $\bar f(0,x)$ and
the assumptions on the weakly decreasing, piecewise $C^1$
speed–density curve implicit in the NFD $Q(\rho)$,  
the macroscopic service rate
\[
  S(\Delta)
  =
  \int_0^\infty
     \min\!\bigl\{V\bigl(\rho(\Delta)\bigr),\,\ell/T_h\bigr\}\,
     \bar f(0,\ell)\,d\ell
\]
is concave in $\Delta$ and has a non–increasing derivative
$S'(\Delta)$.  It is \emph{strictly} concave except possibly on those
intervals of $\Delta$ that correspond to flat sections of
$V(\rho)$ (e.g.\ the free‐flow branch of a triangular NFD).
\end{lemma}

\begin{proof}
Differentiate under the integral and set
$F(x)=\bar F(0,x)$, $S(x)=1-F(x)$:
\[
  S'(\Delta)=V'\rho'\!\int_{0}^{\ell^\star}\!S(\ell)\,d\ell.
\]
Lemma~\ref{lem:IFR_logconcave} gives that $S(\ell)$ is log–concave and
non–increasing, so $\ell\mapsto\!\int_{0}^{\ell}S(s)ds$ is
\emph{concave}.  Because $V'(\rho)<0$ and $\rho'(\Delta)>0$,
$S'(\Delta)$ is strictly decreasing, hence $S''(\Delta)<0$.
\end{proof}

\begin{lemma}[ODE reduction for threshold policies]
\label{lem:ODEcollapse}
Let $\tau(x)$ be any non–decreasing release–time map and set
$u(t)=\int_0^\infty\!\bar f\!\bigl(t,x\bigr)\mathbf 1_{t\ge\tau(x)}\,dx$.
Then the spatial integral $\Delta(t)=\int_0^\infty\!\delta(t,x)\,dx$
solves the scalar stochastic differential equation
\[
  d\Delta(t)=\bigl[u(t)-S\!\bigl(\Delta(t)\bigr)\bigr]dt
            -\underbrace{\!\int_0^{\tau^{-1}(t)}\!\!
               \sigma(x)\bar f(t,x)\,dx}_{=:M(t)}\,dW_t,
\]
where $M$ is square–integrable.  
\end{lemma}

\begin{proof}
Integrate (P1) over $[0,\infty)$ and use
$\delta(t,0)=0$, $\delta(t,x)\to0$ as $x\to\infty$, then insert (P3).
The stochastic integral term arises from integrating $d\bar f$ in (P2),
and its Itô isometry grants square integrability.
\end{proof}

\begin{lemma}[Single zero of the costate]
\label{lem:single_zero}
For the Hamiltonian
$H=\Delta+\lambda\,[u-S(\Delta)]$
associated with the lumped dynamics of Lemma~\ref{lem:ODEcollapse},
the costate $\lambda$ satisfies $\dot\lambda=-1+\lambda S'(\Delta)$,
$\lambda(T)=0$.  Under Lemma~\ref{lem:S_concave}, $\lambda$ crosses
zero at most once, from positive to negative values.
\end{lemma}

\begin{proof}
Concavity of $S$ implies $S'$ is decreasing.
Starting from $\lambda(T)=0$ and integrating backward,
$\dot\lambda=-1$ initially forces $\lambda$ positive.
As time continues backward, $\Delta$ (and thus $S'$)
is non–decreasing, so the damping term
$\lambda S'(\Delta)$ grows, eventually driving
$\lambda$ downward.  Monotonicity of $S'$ rules out a second crossing.
\end{proof}

\begin{proof}[Proof of Proposition~\ref{prop:PDE_IFR_bangbang}]
Lemma~\ref{lem:ODEcollapse} reduces the SPDE control problem to an Itô SDE with linear control influence in the Hamiltonian (the indicator $\mathbf 1_{t\ge\tau(x)}$ multiplies the uncontrolled inflow $\bar f$ but does not appear elsewhere).
Linearity implies that, for any fixed co–state value, the optimal $u(t)$ lies at one of the interval endpoints $\{0,u_{\max}\}$.
Lemma~\ref{lem:single_zero} ensures the costate changes sign exactly once; hence $u^\star$ is bang–bang with a single switch.
The switching time $t_b$ is characterized by
$\lambda(t_b)=0$, which under
$\dot\lambda=-1+\lambda S'(\Delta)$ yields
$S'\!\bigl(\Delta(t_b)\bigr)=1$ and therefore a unique
critical queue level $\Delta^\star=(S')^{-1}(1)$.

Placing the release window (period with $u=u_{\max}$)
before the subsequent hold window (period with $u=0$)
strictly dominates any schedule that postpones the release window:
thanks to concavity of $S$, the queue grows faster when $\Delta$ is small and dissipates faster when $\Delta$ is large, so advancing the release window yields a trajectory $\Delta(t)$ that is everywhere no larger—and somewhere strictly smaller—than that of any "hold–then–release" arrangement.
Consequently the chosen single‐switch strategy is the unique minimizer among bang–bang candidates.
\end{proof}

Under IFR residual-distance distributions, the bang–bang strategy is not merely convenient—it is provably optimal, greatly simplifies field implementation
(one timer suffices), and delivers a transparent rule
that is easy for the public to understand. 

\noindent The optimality of the doubly bang-bang control strategy reduces the computational complexity of the problem substantially. The initially infinite-dimensional optimization problem reduces to a finite-dimensional problem involving only two scalar parameters. Since the convexity of $J(u)$ can be readily established, optimality follows immediately upon identifying the critical points, requiring no further analysis.

To underscore the necessity of the non-decreasing hazard rate assumption for the bang–bang optimality result, we construct a simple deterministic counterexample where this property is violated and the optimal solution departs from the doubly bang–bang structure. Consider the case $\lambda = 0$, so the objective reduces to minimizing the expected total delay, and assume zero volatility ($\sigma = 0$), such that the stochastic demand process reduces to its mean and the problem becomes deterministic. Let the macroscopic speed-density relationship be $V(\rho) = 1 - \rho$, and suppose the initial trip-length distribution is highly fragmented: one third of the vehicles have trip length $1$, one third length $10$, and one third length $19$. That is,
\[
\bar{f}(0,x) = \tfrac{1}{3}\delta_1(x) + \tfrac{1}{3}\delta_{10}(x) + \tfrac{1}{3}\delta_{19}(x).
\]
The total number of vehicles is normalized to $1$, so each cohort contains mass $1/3$.

We claim that the optimal control releases the three cohorts sequentially—first the short-distance group, then the medium-distance group, and finally the long-distance group—each only after the network has fully cleared. This policy invokes origin gating at two distinct points in time and thus violates the doubly bang–bang structure, which restricts the control to a single activation time.

To prove optimality, we first invoke Proposition~1 to exclude any control that separates vehicles within the same trip-length cohort. We then show that mixing different cohorts is also suboptimal. Consider a perturbed policy in which an $\varepsilon$-fraction of vehicles from the medium-distance cohort is released simultaneously with the short-distance group. By symmetry, this argument extends to the case where an $\varepsilon$-fraction of vehicles from the long-distance cohort is shifted to any other release time.

When only one cohort is active, the density is $\rho = 1/3$, so the speed is $V = 2/3$ and travel times are:
\[
\tau_s = \frac{1}{2/3} = 1.5, \qquad \tau_m = \frac{10}{2/3} = 15.
\]
Advancing the departure of an $\varepsilon$-fraction of medium-distance vehicles to the initial release time increases the travel time of the short-distance cohort due to elevated network density, and postpones the clearance time required for the remaining medium-distance vehicles to commence their trips. On the other hand, it decreases the travel time of the reassigned $\varepsilon$-fraction of medium-distance vehicles, which are now admitted earlier. The travel time sensitivity is
\[
\frac{d\tau(x)}{d\rho} = \frac{x}{(1 - \rho)^2},
\]
so at $\rho = 1/3$, we have
\[
\frac{d\tau_s}{d\rho} = \frac{1}{(2/3)^2} = \frac{9}{4}, \qquad \frac{d\tau_m}{d\rho} = \frac{10}{(2/3)^2} = \frac{90}{4}.
\]
The marginal change in delay, given by the total additional time experienced by all vehicles affected, is therefore:
\[
\frac{dD}{d\varepsilon} = \left( \frac{9}{4} + \frac{90}{4} \right) \cdot \tfrac{1}{3} - \tau_s = \frac{99}{4} \cdot \tfrac{1}{3} - 1.5 = \frac{33}{4} - \frac{3}{2} = \frac{27}{4} = 6.75 > 0.
\]
Hence, the perturbation increases total delay, confirming that the sequential release is indeed optimal. This example demonstrates that if the initial trip-length distribution exhibits multiple isolated density masses that are widely separated in terms of distance and associated travel time, then the optimal gating strategy may no longer coordinate a single joint release. Instead, it converges toward a sequence of temporally separated releases optimized for each group individually. In such cases, the optimal control may require multiple activation times, and the doubly bang–bang structure does not hold.

\subsection{Non-Decreasing Hazard Rate for Convex Evacuation Zones}

We conclude this section by demonstrating that the previously identified prerequisite holds for a central class of probability distributions in evacuation modeling. Specifically, the non-decreasing hazard rate of the initial travel distance distribution—required for optimality of the doubly bang-bang control strategy—is satisfied when the affected region is convex and travel origins are approximately uniformly distributed across the entire region. This result applies to both horizontal evacuations within the region and vertical evacuations to the nearest exit.

\begin{proposition}
Let $S \subset \mathbb{R}^2$ be a compact convex set with nonempty interior, and let $p_1, \ldots, p_n \in S$ be fixed points. Define the random variable $D: S \to \mathbb{R}_{\geq 0}$ by
\[
D(x) = \min_{1 \leq i \leq n} \|x - p_i\|,
\]
where $X$ is uniformly distributed on $S$ with respect to the Lebesgue measure. Let $f$ denote the probability density function of $D$ and $F$ its cumulative distribution function. Then the hazard rate function
\[
h(t) = \frac{f(t)}{1 - F(t)}
\]
is non-decreasing on its domain of definition $\{t \geq 0 : F(t) < 1\}$.
\end{proposition}

\begin{proof}
For each $r \geq 0$, define the sublevel set
\[
A(r) = \{x \in S : D(x) \leq r\} = \bigcup_{i=1}^n \{x \in S : \|x - p_i\| \leq r\}.
\]
Since $X$ is uniformly distributed on $S$, we have
\[
F(r) = \mathbb{P}(D \leq r) = \frac{\mu(A(r))}{\mu(S)},
\]
where $\mu$ denotes the Lebesgue measure in $\mathbb{R}^2$.

Since the distance function $d(x, S)$ is Lipschitz continuous and the set $S$ is compact and convex, the function $r \mapsto \mu(A(r))$, which describes the volume of the $r$-neighborhood of $S$, is absolutely continuous. Consequently, the normalized cumulative distribution function $F(r) := \mu(A(r)) / \mu(S)$ is also absolutely continuous. Therefore, its derivative exists almost everywhere and defines the corresponding density function:
\[
f(r) = F'(r) = \frac{1}{\mu(S)} \cdot \frac{d}{dr} \mu(A(r)).
\]

For almost every $r > 0$, the derivative $\frac{d}{dr}\mu(A(r))$ equals the one-dimensional Hausdorff measure (length) of the portion of $\partial A(r)$ that lies in the interior of $S$. Denote this active boundary by
\[
\Gamma(r) = \partial A(r) \cap \text{int}(S),
\]
where $\text{int}(S)$ denotes the interior of $S$. Then
\[
f(r) = \frac{\mathcal{H}^1(\Gamma(r))}{\mu(S)},
\]
where $\mathcal{H}^1$ denotes the one-dimensional Hausdorff measure.

To prove that the hazard rate $h(r) = \frac{f(r)}{1-F(r)}$ is non-decreasing, we establish that $f'(r) \geq 0$ for almost every $r > 0$. This immediately implies $h'(r) \geq 0$ through the identity
\[
h'(r) = \frac{f'(r)(1-F(r)) + f(r)^2}{(1-F(r))^2}.
\]
Since $1-F(r) > 0$ on the domain of $h$ and $f(r)^2 \geq 0$, showing $f'(r) \geq 0$ suffices.

From $f(r) = \frac{\mathcal{H}^1(\Gamma(r))}{\mu(S)}$, we have
\[
f'(r) = \frac{1}{\mu(S)} \cdot \frac{d}{dr}\mathcal{H}^1(\Gamma(r)).
\]

The key geometric observation is that $\Gamma(r)$ is composed of circular arcs of radius $r$, each centered at one of the points $p_i$. As the parameter $r$ increases, these arcs expand outward uniformly in the direction of their outward normal vectors, meaning that each point on $\Gamma(r)$ moves with unit speed in the normal direction. Because the curvature $\kappa(s)$ of a circular arc of radius $r$ is constant and equal to $1/r$, the evolution of the total arc length $\mathcal{H}^1(\Gamma(r))$ is governed by the corresponding curvature-driven flow equation:
\[
\frac{d}{dr} \mathcal{H}^1(\Gamma(r)) = -\int_{\Gamma(r)} \kappa(s) \, ds,
\]
where $\kappa(s)$ denotes the signed curvature at point $s$ along $\Gamma(r)$, with the convention that $\kappa > 0$ for arcs that bend outward, i.e., away from the enclosed region $A(r)$.

For freely expanding circular arcs of radius $r$, we would have $\kappa = 1/r$ uniformly. However, the crucial insight is that interference effects—where circles meet each other or the boundary $\partial S$—cause the active boundary $\Gamma(r)$ to decrease in length more slowly than it would for freely expanding circles. When two circles intersect, their intersection point moves outward, and the total active boundary length is less than the sum of two separate arcs. When a circle meets $\partial S$, the convexity of $S$ ensures that the contact points move along $\partial S$ in a way that reduces the active arc length. At such intersection points, the active boundary may have corners with infinite curvature in the distributional sense, but these contribute zero measure to the curvature integral.

These geometric constraints arising from the convexity of $S$ ensure that the total curvature integral satisfies
\[
\int_{\Gamma(r)} \kappa(s) \, ds \leq \frac{1}{r} \cdot \mathcal{H}^1(\Gamma(r)),
\]
and therefore
\[
\frac{d}{dr}\mathcal{H}^1(\Gamma(r)) \geq -\frac{1}{r} \cdot \mathcal{H}^1(\Gamma(r)).
\]

For a more precise analysis, we can decompose $\Gamma(r) = \Gamma_{\text{free}}(r) \cup \Gamma_{\text{int}}(r)$, where $\Gamma_{\text{free}}(r)$ consists of freely expanding arcs entirely in the interior of $S$, and $\Gamma_{\text{int}}(r)$ represents portions affected by interactions. The evolution equation becomes
\[
\frac{d}{dr}\mathcal{H}^1(\Gamma(r)) = -\frac{1}{r}\mathcal{H}^1(\Gamma_{\text{free}}(r)) + \frac{d}{dr}\mathcal{H}^1(\Gamma_{\text{int}}(r)),
\]
where $\frac{d}{dr}\mathcal{H}^1(\Gamma_{\text{int}}(r)) \geq 0$ due to the geometric effects described above. Since $\mathcal{H}^1(\Gamma_{\text{free}}(r)) \leq \mathcal{H}^1(\Gamma(r))$, we obtain
\[
f'(r) = \frac{1}{\mu(S)} \cdot \frac{d}{dr}\mathcal{H}^1(\Gamma(r)) \geq -\frac{1}{r\mu(S)} \cdot \mathcal{H}^1(\Gamma(r)) + \frac{1}{\mu(S)} \cdot \frac{d}{dr}\mathcal{H}^1(\Gamma_{\text{int}}(r)) \geq 0.
\]

At values of $r$ where a new circle starts contributing to the active boundary or when components merge, the function $r \mapsto \mathcal{H}^1(\Gamma(r))$ may fail to be differentiable. However, such points form a set of measure zero, and $f$ remains absolutely continuous. At these points, we interpret $f'(r)$ in the distributional sense or consider one-sided derivatives, both of which remain non-negative by the same geometric arguments.

Therefore, $f'(r) \geq 0$ for almost every $r > 0$, which implies $h'(r) \geq 0$ almost everywhere, establishing that the hazard rate is non-decreasing.
\end{proof}

The proof exploits the favorable isoperimetric properties of convex sets and the geometric behavior of their boundaries under ball erosion. Non-convex surfaces do not necessarily exhibit a non-decreasing hazard rate, as the following example demonstrates: Consider the dumbbell-shaped surface consisting of two unit disks centered at $(-3,0)$ and $(3,0)$ connected by a thin rectangular corridor $[-3,3]\times[-0.1,0.1]$, with a single exit at position $(-3,0)$. The probability density $f(r)$ corresponds to the length of the iso-distance line defined by $\{x \in S: D(x)=r\}$.

For $r<1$, the iso-distance line is a circular arc centered at $(-3,0)$ with radius $r$, confined to the left disk. Upon reaching $r=1$, the iso-distance line extends into the narrow corridor. For $1<r<5$, the iso-distance line consists of a small arc within the corridor at distance $r$ from the exit, as illustrated in Figure 1a. 

This geometric transformation results in a dramatic drop in the length of the iso-distance line and thus in $f(r)$ for $r>1$, as the iso-distance lines are now confined to the narrow corridor with minimal arc length. Meanwhile, the survival probability $1-F(r)$ decreases smoothly as it includes the substantial probability mass from the right disk. Consequently, the distribution resulting from this geometric configuration exhibits a decreasing hazard rate in the region $1<r<5$, where $f(r)$ remains low while $1-F(r)$ remains substantial due to the unexplored right disk.

A significant operational implication of this observation is that non-convex regions of the evacuated region or clusters in the distribution of origin points can lead to artificial bottlenecks during network evacuation under distance-based control schemes.

\begin{figure}[H]
    \centering
    \includegraphics[width=0.6\textwidth]{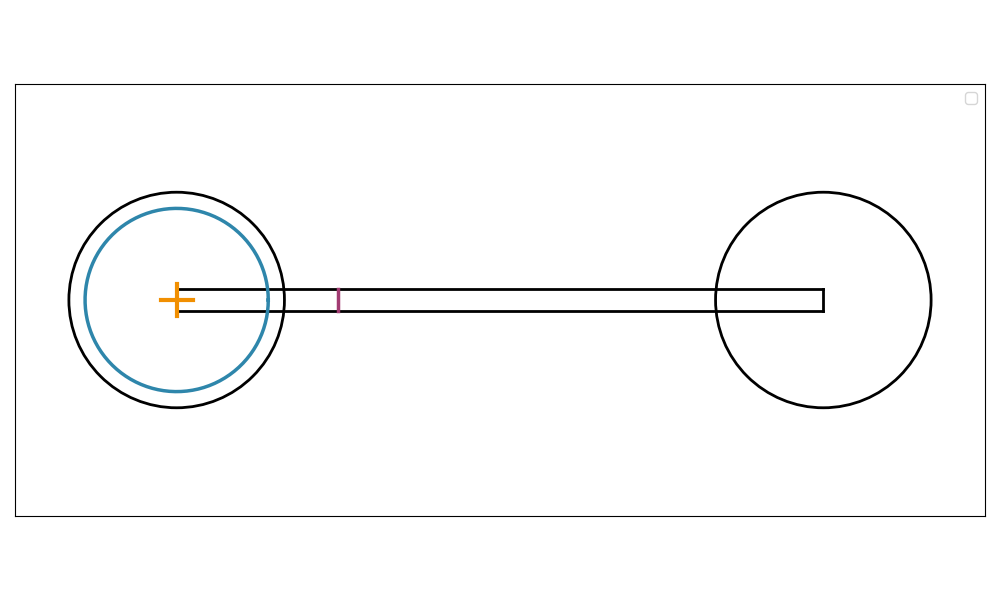}
    \caption{The \textcolor{steelblue}{steel blue} and \textcolor{purple}{purple} lines depict all points at internal distance $r = 0.85$ and $r = 1.5$, respectively, from the exit point (marked in \textcolor{orange}{orange}) within the dumbbell-shaped domain. The comparison between the two iso-distance lines demonstrates the discontiuous drop in the length of the level set as $r$ increases from $0.85$ to $1.5$, transitioning from a circular arc within the left disk to a short vertical segment confined to the narrow corridor.}
    \label{fig:dumbbell}
\end{figure}

\section{Application To Copenhagen Flood Case}
\label{sec:appl}

To parameterize the numerical experiments conducted in this section, we first characterize the key geographic, demographic, and traffic attributes of Amager Island. Amager, situated on the southeastern fringe of Copenhagen, is a low-lying peninsula with an area of roughly \(96\,\text{km}^2\), housing approximately \(230{,}000\) residents - about one-third of Copenhagen’s population. Amager island sits just meters above sea level, making it highly vulnerable to flooding from both coastal storms and heavy rainfall. The island has experienced severe flood events including coastal surges in 1999, 2006, and 2007 that pushed Copenhagen Harbour water levels over 1.3 meters above normal \cite{climatechangepost_coastal}. A major cloudburst in 2011 overwhelmed drainage systems, causing widespread damages exceeding 1 Billon USD, while 2023 brought record-breaking surge levels requiring evacuations \cite{ap_babet2023}.

All motor-vehicle access to the Danish mainland is funneled through four primary bridge exits:
\begin{enumerate}
    \item the four-lane Amager Motorway (E20) bridge at Tårnby,
    \item Langebro, a six-lane urban arterial connecting to the historic city core,
    \item Knippelsbro, a four-lane bridge providing an additional link from central Copenhagen, and
    \item the Kalvebod bridge pair, comprising two adjacent three-lane structures linking the western residential districts to the ring-motorway system.
\end{enumerate}

\noindent In the reservoir approximation employed in this study, the bridges connecting Amager to the city center are modeled as perimeter on-ramps with independent discharge capacities. To estimate their vehicular throughput, we assume that a full contraflow lane-reversal strategy is implemented under the high-urgency evacuation scenario considered. Specifically, all lanes in both directions are allocated to outbound evacuation traffic, except one lane per bridge, which remains reserved for emergency services and towing vehicles. Based on representative street-type-specific discharge capacities — assumed as \(2400\,\text{veh/h/lane}\) for motorway segments and \(2000\,\text{veh/h/lane}\) for urban arterials — the resulting effective bridge capacities are:

\begin{itemize}
    \item \textbf{Amager Motorway (E20)}: \( (4 \times 2 - 1) \times 2400 = 16{,}800 \) vehicles/hour
    \item \textbf{Langebro}: \( (6 \times 2 - 1) \times 2000 = 22{,}000 \) vehicles/hour
    \item \textbf{Kalvebod bridge pair}: \( (3 \times 2 \times 2 - 1) \times 2400 = 31{,}200 \) vehicles/hour
    \item \textbf{Knippelsbro}: \( (3 \times 2 - 1) \times 2000 = 10{,}000 \) vehicles/hour
\end{itemize}

The total perimeter evacuation capacity across all three exits is thus approximated at 80,000 vehicles per hour. These capacities define the "service area" of the generalized Bathtub Model and constrain the aggregate outflow rate achievable under full-network saturation. 

To quantify the transportation infrastructure on Amager island, total lane-kilometers of roadway were calculated using OpenStreetMap (OSM) data accessed through the Overpass API. For the present analysis, various minor road types were excluded, including service roads, unpaved tracks, private access roads, non-asphalted surfaces, bicycle-only paths, and roads with speed limits below 25 km/h. Where OSM data lacked explicit lane count information for road segments, default values were assigned according to highway classification, as detailed in Table~\ref{tab:lane_defaults}.

\begin{table}[H]
\centering
\caption{Default lane counts by highway classification}
\label{tab:lane_defaults}
\begin{tabular}{lr}
\toprule
Highway Type & Default Lanes \\
\midrule
Motorway & 4 \\
Trunk & 2 \\
Primary & 2 \\
Secondary & 2 \\
Tertiary & 2 \\
Unclassified & 2 \\
Residential & 2 \\
Service & 1 \\
\bottomrule
\end{tabular}
\end{table}

Road segment lengths were computed using the Haversine formula to account for Earth's curvature. Following the exclusion of low-speed residential streets, the analysis yielded a total of 2,442.1 lane-kilometers of roadway infrastructure on Amager island. The island geometry is modeled as a circular approximation with an area of 96.29 km$^2$ (corresponding to the actual island area), yielding a radius of approximately 5.54 km. The island's 225,746 inhabitants are assumed to be uniformly distributed across the entire area.

The geometric shape of Amager island was first approximated as an octagon using eight boundary coordinates that capture the island's perimeter. The octagonal approximation was defined by boundary points converted to a projected coordinate system with the octagon's centroid as the origin. The transformation of the bridge exit points onto the idealized disk employed a radial projection method. The centroid $(x_c, y_c)$ of the octagon vertices was calculated in projected coordinates. For each exit point at coordinates $(x_j, y_j)$, the angular position relative to the centroid was computed as:

\begin{equation}
\theta_j = \arctan2(y_j - y_c, x_j - x_c)
\end{equation}

The exit points were then projected radially onto the disk boundary, ensuring all exits lie exactly on the circle's perimeter:

\begin{equation}
(x_{\text{disk}}, y_{\text{disk}}) = (5.54 \cos\theta_j, 5.54 \sin\theta_j)
\end{equation}

This transformation preserves the relative angular positions of exit points while standardizing their radial distance from the island center. The computed angular positions of the bridge endpoints are:
Knippelsbro at $(-0.281,\ 5.533)\,\text{km}$ ($\theta = 92.9^\circ$), 
Langebro at $(-0.505,\ 3.207)\,\text{km}$ ($\theta = 99.0^\circ$), 
Sjællandsbroen at $(-4.555,\ 3.153)\,\text{km}$ ($\theta = 145.3^\circ$), and 
Amagermotorvejen at $(-5.367,\ -1.372)\,\text{km}$ ($\theta = 194.3^\circ$).

To parameterize the MFD and its derived macroscopic speed-density relationship within the generalized bathtub framework, we estimate flow characteristics per lane-kilometer that are representative of the observed mixed roadway composition. Since the analysis excludes low-speed residential streets and primarily retains arterial and motorway infrastructure, we adopt standard values from the literature for urban freeway environments. Specifically, we assume a triangular shape of the local fundamental diagram $q(\rho)$ with a per-lane free-flow speed of $v_f = 65$~km/h, a maximum flow (capacity) of $q_{\max} = 1600$~veh/h/lane, and a jam density of $\rho_j = 120$~veh/km/lane. These values yield a critical density of $\rho_c = q_{\max}/v_f \approx 24.6$~veh/km/lane. The total network-level NFD is obtained by scaling these per-lane values with the aggregated lane-kilometer total reported above. 

To estimate the hydrodynamic progression of floodwater over the island's surface, we employ a simplified storm surge model grounded in the one-dimensional shallow-water equations (SWE), derived by depth-averaging the free-surface Euler equations—which themselves approximate the incompressible Navier-Stokes equations under negligible viscosity—and assuming symmetry in the transverse direction. These equations are widely used in storm surge and inundation modeling due to their ability to capture the essential physics of mass and momentum conservation in thin water layers \cite{Chow1959, Bates2010, Mignot2006, Lehner2006}. In their conservative form, the one-dimensional SWE read
\[
\begin{aligned}
\frac{\partial h}{\partial t} + \frac{\partial (hu)}{\partial x} &= 0, \\
\frac{\partial (hu)}{\partial t} + \frac{\partial}{\partial x} \left( hu^2 + \frac{1}{2} g h^2 \right) &= -gh \frac{\partial z}{\partial x},
\end{aligned}
\]
where \( h(x,t) \) is the water depth, \( u(x,t) \) the depth-averaged velocity, \( z(x) \) the bed elevation, and \( g \approx 9.81\,\mathrm{m/s^2} \) the gravitational acceleration. The first equation expresses conservation of mass, while the second governs momentum balance under hydrostatic pressure.
To develop an analytically tractable inundation timeline for Amager, we model the storm surge as a semi-infinite dam-break originating at the Øresund shoreline—specifically, the eastern rim of the circular reservoir. This semi-infinite dam-break initial value problem represents one of the most fundamental benchmark solutions in flood modeling \cite{Ritter1892,Stoker1957,Toro2001,Aureli2024}. The island’s gentle topography is represented by a uniform ground slope \(z(x)=s\bigl(R-x\bigr)\) that rises from sea level at \(x=R\) to \(z_{\max}=10\;\mathrm{m}\) at \(x=-R\), giving \(s=z_{\max}/(2R)\).  For this dry bed with constant slope, the inviscid SWE admit the Dressler similarity solution: a surge of still-water depth \( H_{0} \) propagates westward with front position
\[
x_{f}(t)=R-\sqrt{gH_{0}}\;t\Bigl[\sqrt{1+2sgt/\!\sqrt{gH_{0}}}-1\Bigr],
\qquad 0\le t\le T_{f},
\]
This solution exploits the self-similar structure of the flow by reducing the original PDE system to an ordinary differential equation in the similarity variable $\eta = \frac{x}{t\sqrt{g H_0}}$. We choose \( H_0 = 3\,\mathrm{m} \), a typical still-water depth associated with millennium-scale flood events. Numerical analysis of the resulting risk field \( R(x, y) \) indicates that complete inundation (\( x_f(T_f) = -R \)) occurs after a flooding duration of \( 7.48\,\mathrm{h} \), with a maximum flooded ratio reaching \( 68.2\% \) of the island surface area. These results align with empirical observations. The first-arrival time of water at any interior point \((x,y)\) is therefore
\[
t_{f}(x,y)=\max\!\Bigl\{0,\,
\frac{R-x}{\sqrt{gH_{0}}}
\Bigl[\sqrt{1+2sg(R-x)/gH_{0}}-1\Bigr]\Bigr\},
\]
so shoreline cells flood first and the western ridge last. A static spatial risk weight is extracted as the reciprocal of this arrival time and normalized over the disc:
\[
R(x,y)=
\frac{t_{f}^{-1}(x,y)}
     {\displaystyle\int_{B_{R}}\!t_{f}^{-1}(x',y')\,dx'\,dy'},
\qquad
\int_{B_{R}}\!R(x,y)\,dx\,dy=1 .
\]
Because \(R\) depends only on position and not on the traffic state, it enters the objective $J(U)$ solely through the risk-weighted delay term, leaving all generalized bathtub dynamics and the analytical results of previous sections unchanged while embedding hydraulically realistic surge timing.

Next, we consider the distribution of the minimum distance to the nearest exit under a mixed density. Let 
$R = 5.54$, $(\varphi_1,\varphi_2,\varphi_3)=(92.9^\circ,145.3^\circ,194.3^\circ)$, and 
$c_i = R(\cos\varphi_i,\sin\varphi_i)$ for $i=1,2,3$. For the sampling density, a point $p=(x,y)$ in the closed disk $B_R(0)$ is drawn from the mixture
\[
f_{P}(x,y)=\lambda\,U(x,y)+(1-\lambda)\,R(x,y),
\]
where $0 \le \lambda \le 1$, $U$ is the uniform density on $B_R(0)$, and $R$ is any non-negative probability density supported on $B_R(0)$, representing the spatial risk field. In polar coordinates $p=(r,\theta)$, this becomes
$f_{r,\theta}(r,\theta)=\frac{\lambda\,r}{\pi R^{2}}+(1-\lambda)\,R(r\cos\theta,r\sin\theta)$
for $0\le r\le R$ and $0\le\theta<2\pi$.

We define the distance variables as
\begin{align*}
d_i(r,\theta) 
  &= \sqrt{r^{2} + R^{2} - 2rR\cos(\theta - \varphi_i)} \text{ and} \\  
D 
  &:= \min_{i=1,2,3} d_i.
\end{align*}

Since $D$ is a measurable function of $p$, its distribution under the mixture is the mixture of the component distributions. The cumulative distribution function follows as a convex combination:
$F_D(d)=\lambda\,F_U(d)+(1-\lambda)\,F_R(d)$,
and similarly for the density:
$f_D(d)=\lambda\,f_U(d)+(1-\lambda)\,f_R(d)$,
where $(F_U,f_U)$ are the uniform-case results already derived and $(F_R,f_R)$ refer to the auxiliary density $R$.

\noindent For any admissible $R(x,y)$, we have the general expression
\begin{align*}
F_R(d) 
&= \iint_{B_R(0)} 
\mathbf{1}_{\left\{ \min_i \|(x,y) - c_i\| \le d \right\}} 
\, R(x,y)\, dx\, dy,
\quad \text{for } 0 \le d \le R.
\end{align*}. If $R$ lacks additional structure, this is the simplest exact form, and numerical quadrature over the disk is then required.

When $R$ is radially symmetric, i.e., \( R(x,y) = \frac{g(r)}{2\pi} \) with normalization
\[
\int_0^{R} g(r)\,r\,dr = 1,
\]
polar integration reduces the double integral to a one-dimensional form:
\[
F_R(d) = \int_{r=R-d}^{R} \frac{g(r)\,r}{2\pi} \, L_{\mathrm{union}}(r,d) \, dr
\quad \text{for } 0 \le d \le R.
\]

The union angle \( L_{\mathrm{union}}(r,d) \) is given by the same expression as in the uniform case:
\[
L_{\mathrm{union}}(r,d) = 6\alpha(r,d)
- \sum_{(i,j)\in\{(1,2),(2,3),(3,1)\}} \max\{0,\, 2\alpha(r,d) - \Delta_{ij}\},
\]
where
\[
\alpha(r,d) =
\begin{cases}
\arccos\left(\dfrac{r^{2} + R^{2} - d^{2}}{2rR}\right), & \text{if } |R - r| \le d \le R + r, \\[6pt]
0, & \text{if } d < |R - r|,
\end{cases}
\]
and \( (\Delta_{12}, \Delta_{23}, \Delta_{31}) \approx (0.914, 0.855, 4.516) \).

We assume that $R$ is sufficiently smooth to permit differentiation under the integral sign. This yields
$f_R(d)=\frac{d}{dd}F_R(d)$,
and $f_D(d)$ follows from the mixture formula above. If $R$ is continuous at the three boundary points $c_i$, the small-cap asymptotic behavior remains quadratic but with a modified coefficient:
$F_D(d)=[3\lambda+2\pi R^{2}(1-\lambda)R(c_i)]\frac{d^{2}}{\pi R^{2}}
+\mathcal{O}(d^{3})$,
while $f_D(d)=\mathcal{O}(d)$. For radial densities $R$, this one-dimensional integral formulation ensures efficient computation of the distribution, making the evaluation of $F_D(d)$ computationally simple.

The risk-adjusted cumulative distribution function of travel distances is shown for the parameters $\lambda = 1$ (corresponding to the uniform distribution), $\lambda = \frac{2}{3}$, and $\lambda = \frac{1}{3}$ in Figure~\ref{fig:min_dist_cdf}. To convert population counts into trip demand, a motorization rate of 0.6 is applied, corresponding to an average of approximately 1.7 passengers per vehicle. This value is slightly above the average values commonly used in the literature but we consider it plausible in the context of an evacuation scenario. According to the previously calculated total capacity of the bridges, this results in a total evacuation time of approximately 95 minutes.

\begin{figure}[H]
    \centering
    \includegraphics[width=0.6\textwidth]{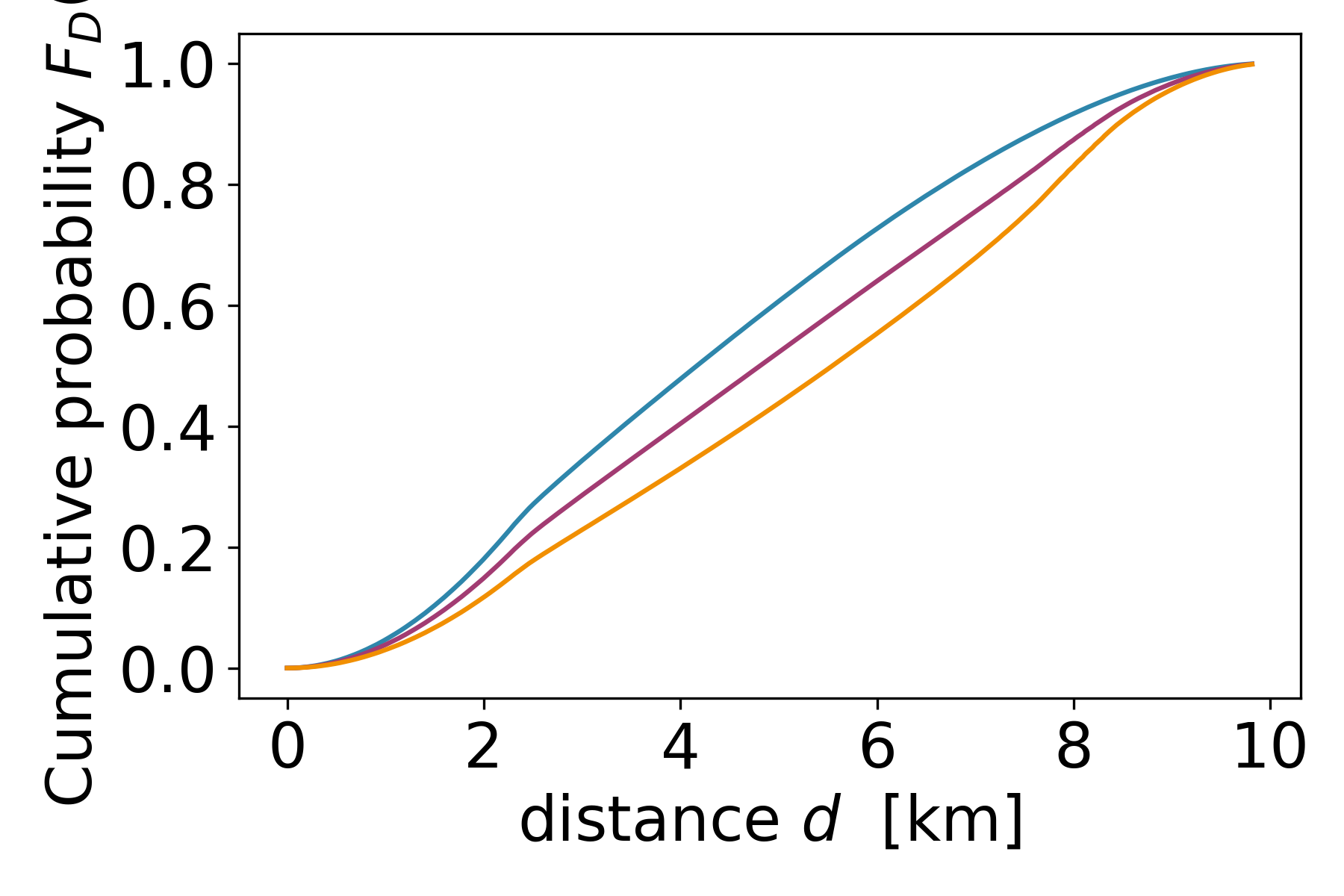}
    \caption{
        Risk-adjusted cumulative distribution functions of travel distances to the nearest exit point, computed for three different values of the mixture parameter $\lambda = 1$ (steel blue), $\lambda = \frac{2}{3}$ (purple) and $\lambda = \frac{1}{3}$ (orange). Each curve shows the probability that an individual’s shortest evacuation path does not exceed a given distance $d$. Higher values of $\lambda$ assign greater weight to points closer to the flood origin and hence farther from the designated evacuation exits, resulting in right-skewed distributions with slower cumulative buildup.
    }
    \label{fig:min_dist_cdf}
\end{figure}

The development of the number of inactive trips follows a geometric Brownian motion, as outlined in \ref{sec:framework}, with parameters $\mu = 1.0$ and $\sigma = 0.03$. This parameterization reflects that in the analyzed scenario, the growth direction of the number of inactive trips is neutral and subject only to moderate fluctuations. Considering that the value of the geometric Brownian motion at time $t$ follows a log-normal distribution $L(\ln(S_0) + (\mu - \sigma^2/2)t, \sigma\sqrt{t})$, the optimal solution can be determined using a simple gradient-based method. The optimal solution was computed for all pairs of $\alpha \in \{0.8, 0.95, 0.99\}$ and $\lambda \in \{1, 2/3, 1/3\}$, with $\lambda = 1/3$ chosen for the weighting parameter of the AVaR. The numerical simulations yielded approximately 75 minutes for complete vehicle evacuation. This falls below the constraint of approximately 95 minutes derived from the maximum capacity of the bridges. Due to the small temporal difference between these two values, we assume that the congestion forming at the bridge approaches does not generate significant spillback effects. Consequently, the traffic dynamics resulting from the Generalized Bathtub Model remain unaffected. The values of the optimal solutions for $t^*$ (rounded to seconds) and $x^*$ are shown in Table~\ref{tab:control_pairs}, while the average values of $J(U)$ per vehicle and the percentage improvement in the objective function compared to the no-control scenario with $t^* = 0$ are presented in Table~\ref{tab:performance_metrics}. The temporal unit of the presented values is minutes and seconds, while the spatial unit is kilometers.

\begin{table}[H]
\centering
\begin{tabular}{lccc}
Alpha $\mid$ Lambda & $1.00$ & $0.67$ & $0.33$ \\ \hline
0.80 & (14$^\prime$16$^{\prime\prime}$, 10.96) & (18$^\prime$08$^{\prime\prime}$, 11.06) & (9$^\prime$11$^{\prime\prime}$, 11.07) \\
0.95 & (10$^\prime$27$^{\prime\prime}$, 10.47) & (10$^\prime$05$^{\prime\prime}$, 10.73) & (9$^\prime$01$^{\prime\prime}$, 11.07) \\
0.99 & (9$^\prime$25$^{\prime\prime}$, 10.03) & (7$^\prime$21$^{\prime\prime}$, 10.68) & (8$^\prime$05$^{\prime\prime}$, 11.08) \\
\end{tabular}
\caption{Optimal control parameters $(t_\star, x_\star)$}
\label{tab:control_pairs}
\end{table}

\begin{table}[H]
\centering
\begin{tabular}{lccc}
Alpha $\mid$ Lambda & $1.00$ & $0.67$ & $0.33$ \\ \hline
0.80 & (8.8, 24.4\%) & (9.4, 23.9\%) & (10.0, 26.7\%) \\
0.95 & (9.4, 24.7\%) & (9.4, 25.7\%) & (11.2, 29.6\%) \\
0.99 & (11.2, 27.4\%) & (17.2, 16.4\%) & (17.8, 25.1\%) \\
\end{tabular}
\caption{Avg. travel time per vehicle (min) and relative improvement}
\label{tab:performance_metrics}
\end{table}

Reducing the risk parameter $\lambda$ causes residents to evacuate earlier from dangerous locations near the dam-break origin at the eastern end of the island. This behavior is reflected in the decreasing evacuation times ($t^*$) and increasing evacuation distances ($x^*$) in the data. An increase in $\alpha$ leads to the same effect, as in extreme scenarios the number of vehicles to be evacuated increases more strongly and delaying their entry is penalized more severely. From Table~\ref{tab:performance_metrics}, it becomes apparent that the average value of the objective function also increases with rising $\alpha$ and decreasing $\lambda$. The increase with rising $\alpha$ occurs due to the consideration of increasingly pessimistic extreme cases. The increase with decreasing $\lambda$ results from the increasing average spatial distance to evacuation bridges. For almost all scenarios, a percentage improvement in $J(U)$ of at least 25\% compared to the no-control scenario was observed.

Finally, we demonstrate the MPC algorithm's behavior using parameters $\lambda = 1$, $\alpha = 0.8$, an update time step of 1 minute, and $\sigma = 0.03$ or $0.1$, respectively. Figures~\ref{fig:mpc_trajectory_1} and~\ref{fig:mpc_trajectory_2} present trajectories of $t^*$ and $x^*$ averaged over 200 realizations. When access-based gating is activated in a run, i.e., when \(t = t^*\) holds and all vehicles are loaded into the network, the values of \(t^*\) and \(x^*\) are assumed to be 0 for the remainder of that run.

The results show that at a confidence level of \(\alpha = 0.8\), hardly any significant differences are detectable between the two values of \(\sigma\). The parameter \(t^*\) decreases with increasing time and converges to zero within approximately 20 minutes, as more vehicles reach their destinations and longer trips can be loaded into the network without causing gridlock-type congestion. For the same reason, \(x^*\) increases over time. Note that the behavior of the control algorithm is exclusively determined by the trajectory of  $t^*$. This is because at time $t^*$, all vehicles are admitted into the network, while $x^*$ is determined statically at the beginning of the algorithm. The plot of the trajectory of $x^*$ serves solely for illustration purposes.

\begin{figure}[H]
    \centering
    \includegraphics[width=0.75\textwidth]{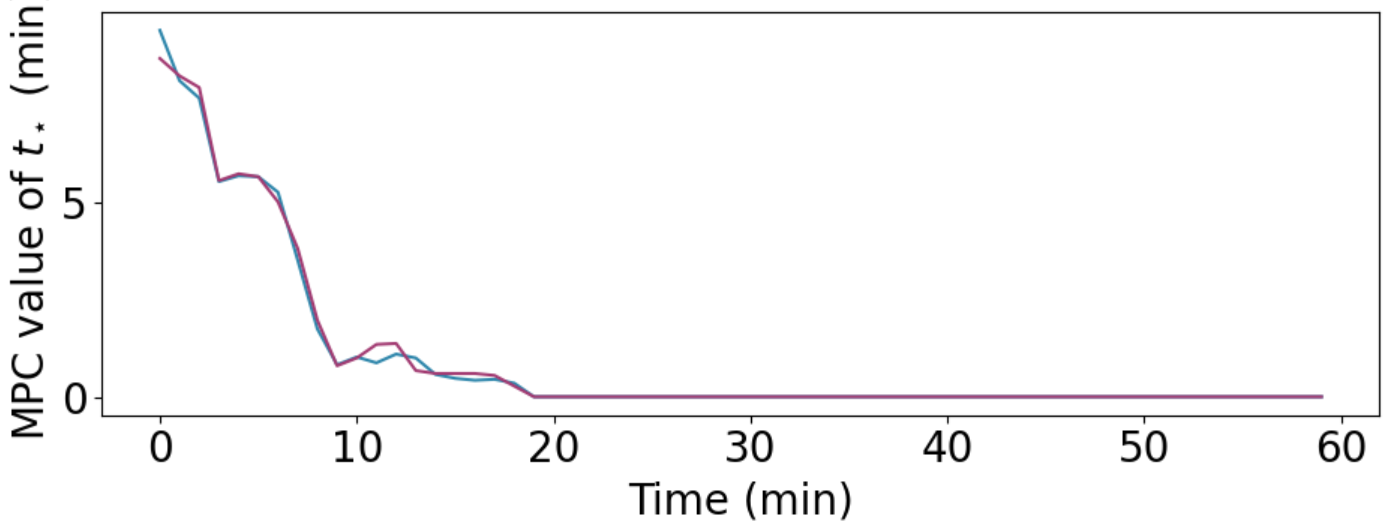}
    \caption{Average trajectory of $t^*$ over time, computed from 200 realizations for $\sigma = 0.03$ (in steel-blue) and $\sigma = 0.1$ (in purple). The plots show a steep decrease in $t^*$ as remaining vehicles can be loaded into the network more optimiztically once the fastest vehicles reach their destination.}
    \label{fig:mpc_trajectory_1}
\end{figure}

\begin{figure}[H]
    \centering
    \includegraphics[width=0.75\textwidth]{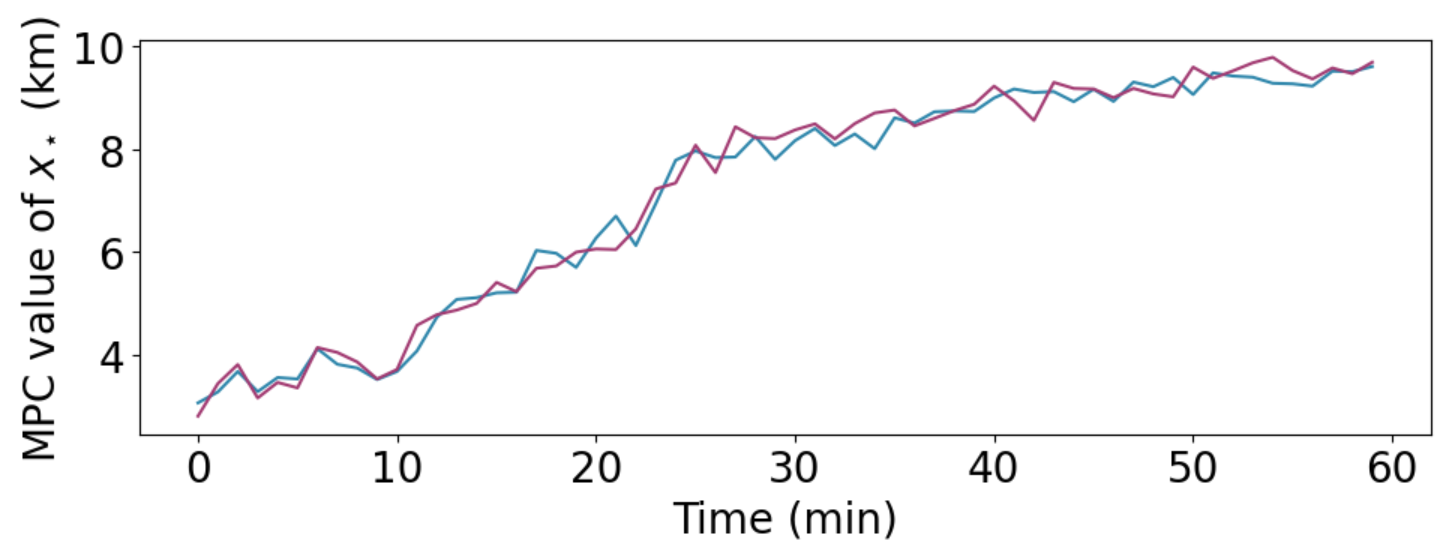}
    \caption{Average trajectory of $x^*$ over time, computed from 200 realizations for $\sigma = 0.03$ and $\sigma = 0.1$. As above, $x^*$ increases as departing vehicles free up network capacity.}
    \label{fig:mpc_trajectory_2}
\end{figure}

\section{Conclusion}
\label{sec:conclusion}

In this paper, we presented a comprehensive risk-aware Model Predictive Control framework for optimal evacuation management in large urban networks under stochastic demand conditions. The framework addresses a critical gap in evacuation planning by providing a mathematically rigorous yet computationally tractable approach to real-time evacuation control that explicitly accounts for evolving threat dynamics and demand uncertainty.
The core contribution lies in the integration of the Generalized Bathtub Model with stochastic control theory to capture network-wide traffic dynamics through the evolution of remaining trip distance distributions. This approach overcomes the computational limitations of traditional single corridor-level models of network flow while maintaining sufficient fidelity to represent heterogeneous trip patterns characteristic of evacuation scenarios. The formulation employs origin gating as the primary control mechanism, implemented through staged departure orders or adaptive ramp metering, offering a practical intervention strategy that can be rapidly deployed without extensive infrastructure modifications.
A key theoretical result demonstrates that under a non-decreasing hazard rate of the initial distance distributions—a condition satisfied by many practical evacuation scenarios including convex evacuation zones with uniformly distributed origins—the optimal control reduces to a computationally efficient bang-bang solution with a single switching time. This finding provides both analytical tractability and operational clarity, suggesting that early heavy release followed by complete throttling represents the optimal evacuation strategy under these conditions. The proof exploits the favorable properties of convex sets and establishes that the "inside-out" evacuation pattern, where households closer to exits depart before those at greater distances, emerges naturally from the optimization framework.
The application to Amager Island's flood evacuation scenario validates the framework's practical utility. Coupling the evacuation model with shallow-water flood simulation based on real bathymetric data, the analysis demonstrates substantial performance improvements through a reduction of more than 25 \% in total evacuation time across almost all tested scenarios.

Several promising directions emerge for future research. First, the integration of more sophisticated time-dependent demand trajectories represents a valuable extension. The current model's assumption that all evacuees prefer immediate departure is defensible for high-urgency scenarios where official evacuation orders are issued, yet real evacuation behavior often exhibits more complex temporal patterns based on threat perception, information availability, and social dynamics. Future work should incorporate behavioral models that capture time-varying evacuation propensity, building upon established research in evacuation decision-making ((23), (24)) and destination choice modeling ((25), (26)), while extending these frameworks to include information diffusion effects and heterogeneous risk perception across population segments.
Second, multi-modal evacuation strategies warrant investigation, particularly the coordination between vehicular evacuation and public transit systems. The framework could be extended to optimize bus routing, shelter capacity allocation, and the temporal coordination between different transportation modes. Third, the incorporation of network topology constraints and capacity degradation under hazard exposure would enhance realism. Current bridge capacities assume perfect infrastructure functionality, whereas floods, incidents, or other adverse conditions may dynamically reduce network capacity.
Additionally, the framework's extension to multi-hazard scenarios presents both theoretical and computational challenges. Different hazard types exhibit distinct spatial and temporal evolution patterns, requiring adaptive risk field formulations and potentially non-stationary stochastic processes for demand modeling. Finally, integration with modern communication technologies and real-time data streams could enable more responsive control policies that leverage GPS tracking, mobile phone data, and social media information to refine evacuation progress estimates and adjust control parameters dynamically.

\newpage

\bibliographystyle{trb}
\bibliography{trb_template}

\end{document}